\documentclass[12pt]{amsart}
\usepackage{amsmath, amssymb, amsthm, bm, mathtools, enumitem, caption}
\usepackage{hyperref}
\usepackage{graphicx, tikz}

\usepackage{comment}

\usepackage[margin=1in]{geometry}
\usepackage{xcolor}
\theoremstyle{plain}
\newtheorem{theorem}{Theorem}[section]

\newtheorem{lemma}[theorem]{Lemma}
\newtheorem{corollary}[theorem]{Corollary}
\newtheorem*{conjecture*}{Conjecture}

\theoremstyle{definition}

\newtheorem{remark}[theorem]{Remark}

\newcommand{\ZZ}{\mathbb{Z}}
\newcommand{\QQ}{\mathbb{Q}}

\DeclareMathOperator{\Gal}{Gal}
\DeclareMathOperator{\Cl}{Cl}

\DeclareMathOperator{\num}{num}
\DeclareMathOperator{\den}{den}

\newcommand{\pif}{\pi} 
\newcommand{\thetaf}{\vartheta} 
\newcommand{\omegapf}{\omega_{\mathrm{pf}}} 

\title{Almost all primes are partially regular}
\date{February 3, 2026}

\newif\ifmanyauthors
\manyauthorstrue 

\ifmanyauthors
  \usepackage[foot]{amsaddr}
  \usepackage{textcomp}
  \newcommand{\dmd}{\ensuremath{\diamond}}

  \author{Evan Chen\textsuperscript{\dag}}
  \email{evan@axiommath.ai}

  \author{Chris Cummins\textsuperscript{*}}
  \email{chris@axiommath.ai}

  \author{Ben Eltschig\textsuperscript{*}}
  \email{ben@axiommath.ai}

  \author{Dejan Grubisic\textsuperscript{*}}
  \email{dejan@axiommath.ai}

  \author{Leopold Haller\textsuperscript{*}}
  \email{leo@axiommath.ai}

  \author{Letong Hong\textsuperscript{\dmd}}
  \email{carina@axiommath.ai}

  \author{Andranik Kurghinyan\textsuperscript{*}}
  \email{andranik@axiommath.ai}

  \author{Kenny Lau\textsuperscript{\dag*}}
  \email{kenny@axiommath.ai}

  \author{Hugh Leather\textsuperscript{*}}
  \email{hugh@axiommath.ai}

  \author{Seewoo Lee\textsuperscript{\dag}}
  \email{seewoo@axiommath.ai}

  \author{Aram Markosyan\textsuperscript{*}}
  \email{am@axiommath.ai}

  \author{Ken Ono\textsuperscript{\dag}}
  \email{ken@axiommath.ai}

  \author{Manooshree Patel\textsuperscript{*}}
  \email{manooshree@axiommath.ai}

  \author{Gaurang Pendharkar\textsuperscript{*}}
  \email{gaurang@axiommath.ai}

  \author{Vedant Rathi\textsuperscript{*}}
  \email{vedant@axiommath.ai}

  \author{Alex Schneidman\textsuperscript{*}}
  \email{alex@axiommath.ai}

  \author{Volker Seeker\textsuperscript{*}}
  \email{volker@axiommath.ai}

  \author{Shubho Sengupta\textsuperscript{\dmd}}
  \email{shubho@axiommath.ai}

  \author{Ishan Sinha\textsuperscript{*}}
  \email{ishan@axiommath.ai}

  \author{Jimmy Xin\textsuperscript{*}}
  \email{jimmy@axiommath.ai}

  \author{Jujian Zhang\textsuperscript{\dag*}}
  \email{jujian@axiommath.ai}

\else

  \author{Evan Chen}
  \email{evan@axiommath.ai}

  \author{Kenny Lau}
  \email{kenny@axiommath.ai}

  \author{Seewoo Lee}
  \email{seewoo@axiommath.ai}

  \author{Ken Ono}
  \email{ken@axiommath.ai}

  \author{Jujian Zhang}
  \email{jujian@axiommath.ai}

\fi

\address{Axiom Math, 124 University Avenue, Palo Alto, CA 94301}

\usepackage[obeyFinal,textsize=tiny,shadow,loadshadowlibrary]{todonotes}
\usepackage{soul}
\newcommand{\added}[1]{\ifmmode{\color{red}#1}\else{\color{red}\ul{#1}}\fi}

\begin{document}

\maketitle

\ifmanyauthors
\begin{center}
  \footnotesize
  Authors are listed alphabetically. \\
  \textsuperscript{\dag}Mathematical contributor,
  \textsuperscript{*}Engineering contributor,
  \textsuperscript{\dmd}Principal investigator.
\end{center}
\fi

\begin{abstract}
For odd primes $p$, we let $K_p:=\mathbb{Q}(\zeta_p)$ be the $p$th cyclotomic
field and let $\omega$ denote its Teichm\"uller character.
For $\alpha>1/2$,  we say that an odd prime $p$ is \emph{partially regular}
if the eigenspaces of the $p$-Sylow
subgroup of $\Cl(K_p)$ under the Galois action vanish for all characters
$\omega^{p-2k}$ with
\begin{equation}\label{range}
2\le 2k \le \frac{\sqrt{p}}{(\log p)^{\alpha}}.
\end{equation}
Equivalently, $p\nmid \num(B_{2k})$ throughout this range.
We prove that a density-one subset of primes is partially regular
in this sense.  By {\it Leopoldt reflection}, this yields a \emph{partial Vandiver Theorem}:
for a density-one set of primes $p$, the even eigenspaces
$A_p(\omega^{2k})$ vanish for all even
$2k$ satisfying \eqref{range}.
This result has consequences for Kubota--Leopoldt $p$-adic $L$-functions,
congruences between cusp forms and Eisenstein series, and $p$-torsion in
algebraic $K$-groups.  The theorem proving partial regularity for
almost all $p$ is fully formalized in Lean/Mathlib
and was produced
automatically by AxiomProver from a natural-language statement of the
conjecture.
\end{abstract}

\section{Introduction}

Let $p$ be an odd prime and set $\zeta_p:=e^{2\pi i/p}$.
Write $K_p:=\QQ(\zeta_p)$ and $\mathcal{O}_{K_p}:=\ZZ[\zeta_p]$ for the ring of integers.
Let $\Cl(K_p)$ denote the ideal class group of $\mathcal{O}_{K_p}$.
Kummer's method for Fermat's Last Theorem (FLT) with exponent $p$
is based on the factorization
\[
x^p+y^p=\prod_{a=0}^{p-1} (x+\zeta_p^a y)
\qquad \text{in }\mathcal{O}_{K_p},
\]
together with factorization properties as dictated by the class
group.
The classical criterion of Kummer,  referred to as $p$-\emph{regularity}, in modern terms,
asserts that
if $p\nmid |\Cl(K_p)|$, then FLT holds for exponent $p$ (for example,
see Kummer's original work \cite{Kummer1850} and modern expositions such as \cite{Washington1997}).

Regularity is a condition that is easily described in terms of Bernoulli numbers.
A prime $p$ is \emph{regular} if $p \nmid B_{2k}$ for all even $2k \in \{2, \dots, p-3\}$.
For irregular primes, however, we may still consider specific indices.
With this as motivation, combined with a theorem of Jensen that asserts that
there are infinitely many irregular primes \cite{Jensen},  we study ``partial regularity".
Siegel's conjecture \cite{Siegel},
that irregular primes occur with asymptotic density
$1-e^{-1/2}\approx 0.393469\dots,$ augments interest in this direction.
We ask whether there are partial regularity results that hold for almost all primes.

To this end,  we say that an odd prime $p$ is \emph{$m$-regular} if
$p\nmid \num(B_{2k})$ for every even integer $2\le 2k\le \min \left (m, p-3\right).$
The main result proved in this note is the following estimate for $M_{\alpha}(p)$-regular primes,
where  for $\alpha>1/2$, we define
\begin{equation}\label{eq:M-of-p}
M_{\alpha}(p) \coloneq \left\lfloor \frac{\sqrt{p}}{(\log p)^{\alpha}}\right\rfloor.
\end{equation}
We prove that a density-one subset of the primes is partially regular.

\begin{theorem}\label{thm:T2}
Fix $\alpha> 1/2$ and define $M_{\alpha}(p)$ as in \eqref{eq:M-of-p}.
Then there exists a constant $C_{\alpha}>0$ such that as $X\rightarrow +\infty$ we have
\[
\#\Bigl\{\,p\le X \text{ prime}:\ p\ \text{is not $M_{\alpha}(p)$-regular}\Bigr\}
\ \le\ C_{\alpha}\,\frac{X}{(\log X)^{2\alpha}}.
\]
In particular, almost every prime is $M_{\alpha}(p)$-regular.
\end{theorem}

\begin{remark}
  \label{rem:ten}
  In fact, it turns out that
  \[ C_\alpha = 10 \]
  is a suitable constant for every $\alpha > 1/2$ in Theorem~\ref{thm:T2} (as well as Corollary~\ref{thm:T1}).
  In particular, the constant
  may be chosen independently of $\alpha$.
\end{remark}

\begin{remark}
The Prime Number Theorem asserts, as $X\rightarrow +\infty,$ that
\[ \pi(X) \coloneq \# \left \{ p\leq X \ : \ p \ {\text {\rm prime}}\right \}\sim \frac{X}{\log X} \]
Therefore, Theorem~\ref{thm:T2} refers to a density-one subset of the prime numbers.
This is what we mean by almost all primes.
\end{remark}

It is natural to ask whether Theorem~\ref{thm:T2} has implications for the
mathematics related to Fermat's Last Theorem.
To this end, we revisit the well-studied Galois refinement of Kummer's criterion.
We let $A_p$ denote the $p$-part of the class group $\Cl(K_p)$. The action of the Galois group
\[ G_p \coloneq \Gal(K_p/\QQ) \cong (\ZZ/p\ZZ)^\times \]
gives the canonical decomposition of $A_p$ into eigenspaces of $\mathbb{Z}_p[G_p]$-modules
\begin{equation}\label{Eigenspaces}
A_p \cong \bigoplus_{i=0}^{p-2} A_p(\omega^i),
\end{equation}
where $\omega: G_p \to \mathbb{Z}_p^\times$ is the Teichm\"uller character,
which sends $a$ to the \emph{unique} $(p-1)$-st root of unity in $\ZZ_p^\times$
satisfying $\omega(a)\equiv a \pmod p$.
These eigenspaces are divided into a ``minus part'' $A_p^-$ (where $i$ is odd)
and the ``plus part'' $A_p^+$ (where $i$ is even).
The latter corresponds to the $p$-part of the class group of the maximal real
subfield $K_p^+ = \mathbb{Q}(\zeta_p + \zeta_p^{-1})$.

The structural understanding of $A_p^-$ is provided by the Herbrand-Ribet Theorem,
which links the non-vanishing of odd eigenspaces to the divisibility of Bernoulli numbers $B_{2k}$.
For even $2k$ in the range $2 \leq 2k \leq p-3$, this is the statement that
\begin{equation}\label{HSystem}
A_p(\omega^{p-2k}) \neq 0 \iff p \mid B_{2k}.
\end{equation}
This result, initiated by Herbrand \cite{Herbrand} and completed by Ribet
\cite{Ribet1976} using the theory of modular forms,
demonstrates that the ``minus'' components are entirely governed by the
Bernoulli numbers and the $p$-adic properties of the Riemann zeta function at
negative odd integers via Euler's identity $\zeta(1-2k)=-B_{2k}/2k$ (for example, see \cite{IrelandRosen}).

While $A_p^-$ is well-understood,
the ``plus'' part $A_p^+$ is the subject of Vandiver's Conjecture (also attributed to Kummer),
which asserts that $p$ does not divide the class number $h_p^+$ of $K_p^+$.
In terms of eigenspaces, the conjecture asserts the following.

\begin{conjecture*}[Kummer-Vandiver \cite{Kummer1849, Vandiver1929}]
If $p$ is an odd prime, then we have
\[ A_p(\omega^{2k}) = 0 \quad \text{for all even } 2k \in \{2, 4, \dots, p-3\}. \]
\end{conjecture*}
\smallskip

Here we consider a weaker form of Vandiver's Conjecture,
which pertains to the triviality of initial ranges of the eigenspaces.
The following result for almost every prime $p$ follows from Theorem~\ref{thm:T2}.

\begin{corollary}\label{thm:T1}
If $\alpha>1/2$, then there is a constant
$C_{\alpha}>0$ such that as $X\rightarrow +\infty$ we have
\[
  \#\Bigl\{\,p\le X \text{ prime}:\  A_p(\omega^{2k})\neq 0 \ \text{for some}\ 2\leq 2k\leq M_{\alpha}(p) \Bigr\}
  \ \le\ C_{\alpha}\,\frac{X}{(\log X)^{2\alpha}}.
\]
In particular, for almost every prime $p$ we have
\[ A_p(\omega^2) \ = \ A_p(\omega^4)\ = \ \dots \ = A_p(\omega^{M_{\alpha}(p)})=0. \]
\end{corollary}
\smallskip

Corollary~\ref{thm:T1} follows from Theorem~\ref{thm:T2},
which is directly related to the odd eigenspaces by the Herbrand-Ribet Theorem.
The link between the two halves $A_p^{-}$ and $A_p^+$ of the class group is
established by a {\it reflection principle}, known as Leopoldt's \textit{Spiegelungssatz} \cite{leopoldt}.
This theorem provides a specialized injection
\[ A_p(\omega^{2k}) \neq 0 \implies A_p(\omega^{p-2k}) \neq 0. \]
This implies that any $p$-divisibility in the even Vandiver components are
reflected as $p$-divisibility in the corresponding odd Herbrand-Ribet components.
This gives a one way relationship between Bernoulli numbers and the Vandiver conjecture,
which can be viewed as a local refinement of the classical notion of $p$-regularity.

\begin{proof}[Deduction of Corollary~\ref{thm:T1} from Theorem~\ref{thm:T2}]
Corollary~\ref{thm:T1} is an immediate consequence of Theorem~\ref{thm:T2} by the
Herbrand--Ribet theorem and Leopoldt's reflection principle
\[ A_p(\omega^{2k}) \neq 0 \implies A_p(\omega^{p-2k}) \neq 0. \qedhere \]
\end{proof}

The result of Corollary~\ref{thm:T1} regarding the triviality of the eigenspaces
$A_p(\omega^{2k})$ has immediate consequences for $p$-adic analytic theory.
To state this, we recall the \textit{Kubota-Leopoldt $p$-adic $L$-function} $L_p(s, \chi)$,
which interpolates the values of Dirichlet $L$-functions $p$-adically.
By the ``Main Conjecture of Iwasawa Theory'' proved by Mazur and Wiles
\cite{MazurWiles} and Leopoldt's reflection principle described below,
the vanishing of the algebraic eigenspace $A_p(\omega^{2k})$ is equivalent to
the associated $p$-adic $L$-function being a unit in the Iwasawa algebra
$\Lambda = \mathbb{Z}_p[[T]]$ when $p \nmid B_{2k}$.
Thus, Corollary~\ref{thm:T1} yields the following analytic nonvanishing theorem for $p$-adic $L$-functions,
which asserts that $\ZZ_p$ is a zero-free region.

\begin{corollary}\label{RH}
  If $\alpha >1/2$ and $M_{\alpha}(p)$ is defined as in \eqref{eq:M-of-p}, then for almost every prime $p$,
  the Kubota-Leopoldt $p$-adic $L$-functions satisfy
  \[
    \big| L_p(s, \omega^{2k}) \big|_p = 1 \quad \text{for all } s \in \mathbb{Z}_p,
  \]
  for all even integers $2 \le 2k \le M_{\alpha}(p)$.
  In particular, these functions have no zeros on $\mathbb{Z}_p$.
\end{corollary}

\begin{remark}
    It is important to note that Corollary~\ref{RH} is qualitatively different (and, in a
$p$-adic sense) and much stronger than the zero-free statement one has in the
complex-analytic setting of the classical Riemann Hypothesis.
    While the Generalized Riemann Hypothesis (GRH) predicts that all non-trivial
    zeros of $L(s, \chi)$ lie on the critical line $\text{Re}(s) = 1/2$,
    the $p$-adic result above asserts that there are no zeros at all in the domain $\ZZ_p$.
    This phenomenon, while striking from a complex-analytic perspective, is standard in $p$-adic analysis.
    It corresponds to the case where the characteristic polynomial in Iwasawa theory is a unit (degree 0).
    This reflects the fact that ``most" primes are expected to be regular (or satisfy Vandiver's conjecture),
    in which case the associated arithmetic objects are trivial and the analytic
    functions are invertible units.
\end{remark}
\begin{proof}[Proof of Corollary~\ref{RH}]
By the standard control theorem in cyclotomic Iwasawa theory, the $\omega^{p-2k}$-eigenspace of the
unramified Iwasawa module over $K_\infty/K$ has $\Gamma$-coinvariants canonically identified with
$A_p(\omega^{p-2k})$ (see e.g. Washington~\cite[Ch.~13]{Washington1997}).  Hence
$A_p(\omega^{p-2k})=0$ forces this eigenspace to be zero by Nakayama's lemma.
By Mazur--Wiles~\cite{MazurWiles}, the cyclotomic main conjecture identifies the characteristic
ideal of this eigenspace with the principal ideal generated by the Kubota--Leopoldt $p$-adic
$L$-function attached to $\omega^{2k}$ (equivalently, to $\omega^{1-(p-2k)}$), so that vanishing of the
module is equivalent to this $p$-adic $L$-function being a unit of $\Lambda$.
Therefore, its values on $\ZZ_p$ are $p$-adic units, i.e.
$|L_p(s,\omega^{2k})|_p=1$ for all $s\in\ZZ_p$.
\end{proof}

\medskip
We record several further ways in which Theorem~\ref{thm:T2}, through the ubiquity of Bernoulli
numbers, governs phenomena across number theory and arithmetic geometry.
Classically, Bernoulli numbers first arise in formulas for sums of
powers of integers via Faulhaber's formula (see Chapter 15 of \cite{IrelandRosen}), but their deeper
significance emerges through their appearance in special values of
$L$-functions. In particular, the values of the Riemann zeta function at
negative odd integers are given by
\[
\zeta(1-2k) = -\frac{B_{2k}}{2k},
\]
so that divisibility properties of Bernoulli numerators encode subtle
information about these special values. This connection plays a central
role in algebraic number theory, especially in Iwasawa theory and the
study of cyclotomic fields as mentioned earlier (also see \cite[Chapter~5]{Washington1997}).

A second, striking appearance of Bernoulli numbers occurs in the theory
of modular forms. Certain congruences between Eisenstein series and
cuspidal Hecke eigenforms are controlled precisely by divisibility of
Bernoulli numbers. The most famous example is Ramanujan's tau-function congruence
\[
\tau(n) \equiv \sum_{d\mid n}d^{11} \pmod{691},
\]
relating the Fourier coefficients of the weight~$12$ cusp form
$\Delta(q)$ to those of the Eisenstein series $E_{12}(q)$. This congruence
reflects the fact that
\[
B_{12} = -\frac{691}{2730},
\]
so that the prime $691$ divides the numerator of $B_{12}$. More
generally, as explained by Swinnerton-Dyer \cite{SwinnertonDyer}, primes dividing numerators of
Bernoulli numbers give rise to congruences between even-weight Eisenstein
series and cuspidal eigenforms, revealing deep links between special
values of $L$-functions, rational torsion on elliptic curves,
and the structure of Hecke algebras (see \cite{DiamondShurman2005, Mazur1977, OnoCBMS}).

Finally, Bernoulli numbers also govern torsion phenomena in algebraic
$K$-theory. Algebraic $K$-groups provide a systematic way to measure the
failure of unique factorization and related structural properties in
rings, generalizing classical invariants such as the unit group and the
class group. In the case of the integers $\ZZ$, deep conjectures of
Lichtenbaum and Quillen, proved through the work of Borel and others,
relate the $p$-torsion in higher $K$-groups $K_{4k-2}(\ZZ)$ to divisibility
of Bernoulli numbers. Thus, just as in cyclotomic class groups and
modular forms, primes dividing Bernoulli numerators control subtle
torsion phenomena in $K$-theory. We refer the reader to
\cite[Chapter~18]{Weibel2013} for an accessible introduction to these ideas.

We record a convenient four-part corollary illustrating immediate consequences of Theorem~\ref{thm:T2}.

\begin{corollary}\label{cor:three}
Fix $\alpha>1/2$ and let $M_{\alpha}(p)$ be as in \eqref{eq:M-of-p}.
Then for almost every  prime $p$, the following hold simultaneously
for every integer $1 \le 2k \le M_{\alpha}(p)$:
\begin{enumerate}[label=\textnormal{(\arabic*)}, leftmargin=2.3em]
\item The congruence
\[
\sum_{a=1}^{p-1} a^{2k}\ \equiv\ B_{2k}\,p \pmod{p^2}
\]
holds in $\ZZ/p^2\ZZ$ (interpreting $B_{2k}$ via its reduction modulo $p^2$).
\item
We have that $\zeta(1-2k)$ is a $p$-adic unit.
\item
There is no normalized cuspidal Hecke eigenform $f(q)$ of level $1$ and weight $2k$
whose Fourier expansion satisfies
\[ f(q) \equiv \sum_{n=1}^{\infty}\sum_{d\mid n}d^{2k-1}q^n\pmod p. \]
\item
The group $K_{4k-2}(\ZZ)$ has no $p$-torsion.
\end{enumerate}
\end{corollary}

\begin{remark}[Generalized Bernoulli numbers]
Let $\chi$ be a Dirichlet character modulo $N$.
The generalized Bernoulli numbers $B_{m,\chi}$ are defined by
\[
\sum_{a=1}^{N} \chi(a)\,\frac{t e^{at}}{e^{Nt}-1}
=\sum_{m=0}^{\infty} B_{m,\chi}\,\frac{t^m}{m!},
\]
and satisfy $L(1-m,\chi)=-B_{m,\chi}/m$ for $m\ge 1$ (see \cite[Ch.~4]{Washington1997}).
A {\it mutatis mutandis} version of Theorem~\ref{thm:T2} can be proved for the
$p$-divisibility of numerators of $B_{m,\chi}$ in suitable ranges, yielding
corresponding ``almost all primes'' corollaries for congruences of cusp forms
with Nebentypus, special values of $p$-adic $L$-functions, and torsion in
$K$-groups of rings of integers.
\end{remark}

\begin{remark}[Autonomous proof and Lean verification]
This work is a case study and test case for AxiomProver, an AI tool currently under development.
We asked the system to prove Theorem~\ref{thm:T2}, and it
generated a Lean/Mathlib statement and a fully verified proof.
Using that formal development as a reference point,
we prepared the exposition in the main text for a mathematical audience, aiming to supply context,
motivation, and a streamlined derivation that can be read independently of the Lean code.
\end{remark}

\medskip

This paper is organized as follows.
In Section~\ref{Section2}, we recall  classical facts about Bernoulli numbers,
and in Section 3 we prove Theorem~\ref{thm:T2}.
In Section~\ref{Section4} we discuss the formalization and Lean verification of
the proof of Theorem~\ref{thm:T2}, including weblinks to the code and artifacts.

\section*{Acknowledgements}
The authors thank Ken Ribet, Ashvin Swaminathan and Ila Varma for comments on an
earlier version of this paper.
\ifmanyauthors
\else
This paper describes a test case for AxiomProver,
an autonomous system that is currently under development.
The project engineering team is
Chris Cummins,
Ben Eltschig,
GSM,
Dejan Grubisic,
Leopold Haller,
Letong Hong (principal investigator),
Andranik Kurghinyan,
Kenny Lau,
Hugh Leather,
Aram Markosyan,
Manooshree Patel,
Gaurang Pendharkar,
Vedant Rathi,
Alex Schneidman,
Volker Seeker,
Shubho Sengupta (principal investigator),
Ishan Sinha,
Jimmy Xin,
and Jujian Zhang.
\fi

\section{Nuts and Bolts}\label{Section2}

We first recall the relevant preliminaries for the proof of Theorem~\ref{thm:T2}.
Throughout, we fix $\alpha>1/2$ and set $M_{\alpha}(p)=\lfloor \sqrt{p}/(\log p)^{\alpha}\rfloor$.
For an integer $m\ge 1,$ we define
\begin{equation}\label{Pm}
P_m \coloneq \prod_{k=1}^{m} \bigl|\num(B_{2k})\bigr|\ \in\ \ZZ^{+}.
\end{equation}
Let $\omegapf(n)$ denote the number of \emph{distinct} prime divisors of a positive
integer $n$:
\[
\omegapf(n) \coloneq \#\{\ell \text{ prime}:\ \ell\mid n\}.
\]
The next lemma bounds the number of $m$-regular primes, when $m$ is a fixed positive integer.

\begin{lemma}\label{lem:primes-divide-product}
Let $p\ge 5$ be prime and $m\ge 1$.
If $p$ is not $m$-regular,
then $p\mid P_m$. In particular, we have
\[
\#\{\,p\le X\text{ prime}:\ p\text{ is not $m$-regular}\}\ \le\ \omegapf(P_m).
\]
\end{lemma}

\begin{proof}
If $p$ is not $m$-regular, there exists $2\le 2k\le \min(m,p-3)$ such that
$p\mid \num(B_{2k})$. By definition of $P_m$, the factor $|\num(B_{2k})|$ appears in
the product, hence $p\mid P_m$. The counting inequality follows because distinct
primes $p\le X$ dividing $P_m$ form a subset of the set of all distinct primes
dividing $P_m$, whose cardinality is $\omegapf(P_m)$.
\end{proof}

We first recall two standard facts about Bernoulli numbers.

\begin{lemma}[Euler's formula for $\zeta(2k)$]\label{lem:euler-zeta}
For each integer $k\ge 1$, we have
\[
\zeta(2k)=(-1)^{k+1}\,\frac{(2\pi)^{2k}}{2(2k)!}\,B_{2k} \ \qquad {\text and}\ \qquad
\zeta(1-2k)=-\frac{B_{2k}}{2k}.
\]
\end{lemma}

\begin{proof}
This is classical; see \cite[Ch.~12, Thm.~12.17]{Apostol1976}.
\end{proof}

\begin{lemma}[von Staudt--Clausen]\label{lem:vsc}
For each integer $k\ge 1$, we have
\[
B_{2k}+\sum_{\substack{\ell\ \mathrm{prime}\\ (\ell-1)\mid 2k}}\frac{1}{\ell}\ \in\ \ZZ.
\]
In particular, we have that
\[
\den(B_{2k})=\prod_{\substack{\ell\ \mathrm{prime}\\ (\ell-1)\mid 2k}}\ell, \label{eqn:lean-mess}
\]
which implies that
\begin{equation}
\label{eqn:vsc_consequence}
    \den(B_{2k})\le \prod_{\ell\le 2k+1}\ell.
\end{equation}
\end{lemma}

\begin{proof}
See \cite[Ch.~12, Exercise~12]{Apostol1976} for a proof.
If $(\ell-1)\mid 2k$ then $\ell\le 2k+1$, so the stated upper bound follows.
\end{proof}

We use this lemma to derive an upper bound for the numerators of Bernoulli numbers,
when expressed in lowest terms.

\begin{lemma}\label{lem:num-growth}
There is an absolute constant $C_1>0$ such that for every integer $k\ge 1$, we have
\[
\log\bigl|\num(B_{2k})\bigr|\ \le\ C_1\,k\log(2k).
\]
Consequently, there is an absolute constant $C_2>0$ such that for all $m\ge 2$,
\[
\log P_m\ \le\ C_2\,m^2\log m.
\]
\end{lemma}

\begin{proof}
Fix $k\ge 1$.
By Lemma~\ref{lem:euler-zeta}, we have
\[
|B_{2k}|=\frac{2(2k)!\,\zeta(2k)}{(2\pi)^{2k}}
\le \frac{2(2k)!\,\zeta(2)}{(2\pi)^{2k}}
< \frac{4(2k)!}{(2\pi)^{2k}},
\]
since $\zeta(2)=\pi^2/6<2$.
Thanks to Lemma~\ref{lem:vsc}, denominators of Bernoulli numbers are square-free, and satisfy
\[
\den(B_{2k})\le (2k+1)!.
\]
Writing $B_{2k}=\num(B_{2k})/\den(B_{2k})$ in lowest terms gives
$|\num(B_{2k})|=|B_{2k}|\den(B_{2k})$, hence
\[
|\num(B_{2k})|
\le \left(\frac{4(2k)!}{(2\pi)^{2k}}\right) (2k+1)!.
\]
Taking logs and using that $-2k\log(2\pi)<0$, we may discard the negative term to obtain
\[
\log\bigl|\num(B_{2k})\bigr|
\le \log 4 + \log\bigl((2k)!\bigr) + \log\bigl((2k+1)!\bigr).
\]
Using the elementary bound $\log(n!)=\sum_{j=1}^n \log j \le n\log n$, we get
\[
\log\bigl((2k)!\bigr)\le 2k\log(2k)
\quad\text{and}\quad
\log\bigl((2k+1)!\bigr)\le (2k+1)\log(2k+1).
\]
Moreover, since $2k+1\le 4k$, we have $\log(2k+1)\le \log(4k)=\log(2k)+\log 2\le 2\log(2k)$,
and also $2k+1\le 3k$ for $k\ge 1$. Hence
\[
\log\bigl((2k+1)!\bigr)\le (2k+1)\log(2k+1)
\le (2k+1)\log(4k)
\le 3k\cdot 2\log(2k)=6k\log(2k).
\]
Putting these estimates together yields
\[
\log\bigl|\num(B_{2k})\bigr|
\le \log 4 + 2k\log(2k) + 6k\log(2k)
\le C_1\,k\log(2k)
\]
for a suitable absolute constant $C_1>0$.

For the second claim, summing over $1\le k\le m$ and using $\log(2k)\le \log(2m)$ gives
\[
\log P_m=\sum_{k=1}^m \log\bigl|\num(B_{2k})\bigr|
\le C_1 \log(2m)\sum_{k=1}^m k
= \frac{C_1}{2}m(m+1)\log(2m)
\le C_2\,m^2\log m
\]
for all $m\ge 2$ and a suitable absolute constant $C_2>0$.
\end{proof}

In the following lemma, we apply the Prime Number Theorem.
It is a standard consequence of the asymptotic growth of the primorial.

\begin{lemma}\label{lem:omega-bound}
There exist absolute constants $C_3>0$ and $N_0\ge 3$ such that for every integer
$n\ge N_0$, we have
\[
\omegapf(n)\ \le\ C_3\,\frac{\log n}{\log\log n}.
\]
\end{lemma}

\begin{proof}
Let $t=\omegapf(n)$ and list the distinct prime divisors of $n$ in increasing order:
$q_1<q_2<\cdots<q_t$. Then
\[
n\ \ge\ q_1q_2\cdots q_t\ \ge\ 2\cdot 3\cdot 5\cdots p_t,
\]
where $p_t$ is the $t$th prime and the last product is the $t$th primorial
$p_t^\# \coloneq \prod_{j=1}^{t} p_j$.
Taking logs yields
\[
\log n\ \ge\ \log(p_t^\#)=\thetaf(p_t),
\]
where $\thetaf(x)=\sum_{p\le x}\log p$ is Chebyshev's $\theta$-function.

The Prime Number Theorem is equivalent to $\thetaf(x)\sim x$ as $x\to\infty$
(see \cite[Ch.~4]{Apostol1976}).
In particular, there exists $x_0\ge 2$ such that for all $x\ge x_0$,
\begin{equation}\label{eq:theta-lower}
\thetaf(x)\ \ge\ \frac{x}{2}.
\end{equation}
Likewise, the Prime Number Theorem implies $\pif(x)\sim x/\log x$, hence there
exists $x_1\ge 3$ such that for all $x\ge x_1$, we have
\begin{equation}\label{eq:pi-upper}
\pif(x)\ \le\ 2\,\frac{x}{\log x}.
\end{equation}
Fix $t_0$ so that $p_t\ge \max(x_0,x_1)$ for all $t\ge t_0$.

Now assume $t\ge t_0$.
Applying \eqref{eq:pi-upper} at $x=p_t$ gives
$t=\pif(p_t)\le 2p_t/\log p_t$, hence
$p_t\ge \tfrac{t}{2}\log p_t \ge \tfrac{t}{2}\log t$ (since $p_t\ge t+1$).
Applying \eqref{eq:theta-lower} at $x=p_t$ yields
\[
\thetaf(p_t)\ \ge\ \frac{p_t}{2}\ \ge\ \frac{t}{4}\log t.
\]
Therefore, for $t\ge t_0$, we have
\[
\log n \ \ge\ \thetaf(p_t)\ \ge\ \frac{t}{4}\log t.
\]
Write $L=\log n$.
For $n$ sufficiently large (equivalently $L$ sufficiently large), the inequality
$L\ge \tfrac{t}{4}\log t$ forces $t\le 8L/\log L$ (a standard inversion of
$t\log t\ll L$).
Concretely: choose $L_0$ large enough that $\log(8L/\log L)\ge \tfrac12\log L$
for all $L\ge L_0$. If $t>8L/\log L$ and $L\ge L_0$, then
$t\log t \ge (8L/\log L)\cdot (\tfrac12\log L)=4L$, contradicting $t\log t\le 4L$.
Thus $t\le 8L/\log L$ for all $L\ge L_0$.

Finally, absorb the finitely many integers $n$ with $\log n<L_0$ into the constant.
That is, define $N_0=\left\lceil e^{L_0}\right\rceil$ and take
\[
C_3 \coloneq \max\!\left( 8,\ \max_{3\le n<N_0}\omegapf(n)\frac{\log\log n}{\log n}\right),
\]
which is finite.
Then $\omegapf(n)\le C_3\,\frac{\log n}{\log\log n}$ holds for all $n\ge N_0$.
\end{proof}

\section{Proof of Theorem~\ref{thm:T2}}\label{Section3}
Fix $\alpha>1/2$.
Let $X\ge 3$ and set
\[
m_X \coloneq \left\lfloor \frac{\sqrt{X}}{(\log X)^{\alpha}}\right\rfloor.
\]
We first reduce the varying range $M_{\alpha}(p)$ to the uniform range $m_X$.
Consider the function
\[ f(t)=\frac{\sqrt{t}}{(\log t)^{\alpha}} \]
for $t>1$.
A direct calculation gives
\[
f'(t)=\frac{1}{t^{1/2}(\log t)^{\alpha}}\left(\frac12-\frac{\alpha}{\log t}\right),
\]
so $f$ is increasing for all $t>e^{2\alpha}$.
Hence, for every prime $p$ with $e^{2\alpha}\le p\le X$, we have
\[
M_{\alpha}(p)=\lfloor f(p)\rfloor\ \le\ \lfloor f(X)\rfloor = m_X.
\]
Therefore, if such a prime $p$ is \emph{not} $M_{\alpha}(p)$-regular, it is also not
$m_X$-regular.
The remaining primes $p<e^{2\alpha}$ form a finite set depending only on $\alpha$; we absorb
their contribution into the final constant.

It remains to bound the number of primes $p\le X$ which are not $m_X$-regular.
By Lemma~\ref{lem:primes-divide-product},
\[
\#\{p\le X:\ p\ \text{not $m_X$-regular}\}\ \le\ \omegapf(P_{m_X}).
\]
For $X$ large enough we have $P_{m_X}\ge N_0$ (where $N_0$ is from Lemma~\ref{lem:omega-bound}),
so Lemma~\ref{lem:omega-bound} and Lemma~\ref{lem:num-growth} yield
\[
\omegapf(P_{m_X})
\le C_3\,\frac{\log P_{m_X}}{\log\log P_{m_X}}
\le C_3\,\frac{C_2\,m_X^2\log m_X}{\log\log P_{m_X}}.
\]
We now give a concrete lower bound for $\log\log P_{m_X}$ in terms of $\log m_X$.
For any integer $m\ge 1,$ we have
\[
P_m=\prod_{k=1}^m |\num(B_{2k})|\ \ge\ |\num(B_{2m})|\ \ge\ |B_{2m}|,
\]
since $\den(B_{2m})\ge 1$.
By Euler's formula (i.e., Lemma~\ref{lem:euler-zeta}) and the trivial inequality $\zeta(2m)>1$, we have
\[
|B_{2m}|=\frac{2(2m)!\,\zeta(2m)}{(2\pi)^{2m}}
\ \ge\ \frac{2(2m)!}{(2\pi)^{2m}}.
\]
Taking logarithms and using the elementary Stirling lower bound (for example, see \cite{Robbins})
\[ (2m)!\ge (2m/e)^{2m}, \]
we obtain
\[
\log P_m\ \ge\ \log|B_{2m}|
\ \ge\ \log 2 + \log((2m)!) - 2m\log(2\pi)
\ \ge\ \log 2 + 2m\log(2m) - 2m - 2m\log(2\pi).
\]
In particular, there exists an absolute integer $m_0\ge 2$ such that for all $m\ge m_0$ we have
\[
\log P_m \ge m\log m,
\]
and hence
\[
\log\log P_m \ge \log(m\log m)\ge \log m.
\]
Since $m_X\to\infty$ with $X$, for all sufficiently large $X$ we have $m_X\ge m_0$, and therefore
\[
\log\log P_{m_X}\ge \log m_X.
\]
Substituting this into the previous bound gives, for all sufficiently large $X$,
\[
\omegapf(P_{m_X})
\le C_3\,\frac{C_2\,m_X^2\log m_X}{\log\log P_{m_X}}
\le C_2C_3\,m_X^2.
\]
Finally, $m_X^2\le X/(\log X)^{2\alpha}$, so for all sufficiently large $X$, we have
\[
\#\{p\le X:\ p\ \text{not $M_{\alpha}(p)$-regular}\}\ \le\ \omegapf(P_{m_X})
\le (C_2C_3)\,\frac{X}{(\log X)^{2\alpha}}.
\]
Enlarging the constant to absorb the finitely many exceptional primes $p<e^{2\alpha}$ and the finitely many
remaining values of $X$ proves the theorem.
\qed

\section{AxiomProver's autonomous proof and Lean verification}\label{Section4}
Here, we provide context for this project as well as the protocol used for
formalization and Lean verification.

\subsection{A case study}
To the best of our knowledge, Theorem~\ref{thm:T2} and Corollary~\ref{thm:T1} are new.
The methods are standard for domain experts in analytic number theory (i.e., Bernoulli number congruences,
Euler's formula for $\zeta(1-2k)$, and sieve methods).
We asked whether AxiomProver, an AI tool currently under development,
can autonomously derive a correct proof of an analytic number-theoretic estimate from its statement.

What did we learn?
AxiomProver successfully autonomously proved and Lean verified Theorem~\ref{thm:T2},
the main engine of this paper, given only the natural-language statement of the claim.
No human assistance was provided.

At first glance, the proofs found by AxiomProver do not resemble the narrative presented in this paper.
Turning a Lean file into a human-readable proof can be hard
because Lean is written as code for a type-checker,
not as an explanation for a reader.
It makes all the “obvious” bookkeeping explicit (rewrite steps, coercions, side conditions, case splits),
and it often follows whatever lemmas and tactic-driven routes are most convenient
for the library rather than the most illuminating conceptual path.

A human mathematician can usually condense this dramatically by
relying on shared historical context\footnote{For example,
  \texttt{solution.lean} consists of $\sim 3500$ lines of code,
  and more than $1000$ of them correspond to establishing the conclusion
  \eqref{eqn:vsc_consequence} of Lemma~\ref{lem:vsc}, the classical von Staudt-Clausen Theorem.
  In our narrative, we simply cite this classical result.},
standard arguments, and informal identifications that Lean cannot take for granted.
As a result, writing a paper from Lean files is not a matter of reformatting.
The authors must digest the formal script, reconstruct the underlying ideas,
and then translate the code into a narrative that highlights the key insights
while safely omitting the routine details Lean had to spell out.

\begin{remark}
To make the scope precise, we emphasize that the AI system was not asked to reprove deep external
theorems such as cyclotomic descent or the Herbrand--Ribet Theorem. These are cited from the literature.
The case study focuses on the proof of Theorem~\ref{thm:T2}.
\end{remark}

\subsection{AxiomProver Protocol}
Here we describe the protocol we employed using AxiomProver to
autonomously prove and verify Theorem~\ref{thm:T2} in Lean (see \cite{Lean, Mathlib2020}),
the main engine in the paper apart from the seminal work of Euler, Kummer, Herbrand, Leopoldt, Mazur-Wiles,
and Ribet that we cite and employ.
All of the mathematics required for the formalization of Theorem~\ref{thm:T2} are given in Sections 2 and 3.

\subsection*{Process}
The formal proofs provided in this work were developed and verified using Lean \textbf{4.26.0}.
Compatibility with earlier or later versions is not guaranteed due to the
evolving nature of the Lean 4 compiler and its core libraries.
The relevant files are all posted in the following repository:
\begin{center}
  \url{https://github.com/AxiomMath/partial-regularity}
\end{center}
The input files were
\begin{itemize}
  \item \texttt{problem.tex}, the problem statement in natural language
  \item a configuration file \texttt{.environment} that contains the single line
  \begin{quote}
    \slshape
    lean-4.26.0
  \end{quote}
  which specifies to AxiomProver which version of Lean should be used.
\end{itemize}
Given these two input files,
AxiomProver autonomously provided the following output files:
\begin{itemize}
  \item \texttt{problem.lean}, a Lean 4.26.0 formalization of the problem statement; and
  \item \texttt{solution.lean}, a complete Lean 4.26.0 formalization of the proof.
\end{itemize}
After AxiomProver generated a solution, the human authors wrote this paper
(without the use of AI) for human readers.
Indeed, as mentioned above, a research paper is a narrative designed to communicate ideas to humans,
whereas  Lean files are designed to satisfy a computer kernel.
It is interesting to note that AxiomProver independently derived the von Staudt Clausen theorem in this run.

\subsection*{An additional Lean proof of Remark~\ref{rem:ten}}
In the main run above, AxiomProver did not receive any ``hints''
and was required to produce the entire proof end-to-end.

However, besides the end-to-end run above, we conducted internal experiments
where we provided AxiomProver with additional resources,
such as the formalization of the prime number theorem
maintained by Alex Kontorovich \cite{KontrovichPNT}
(although AxiomProver did not end up needing this capability).
In one of these experiments, AxiomProver found that
actually $C_{\alpha} = 10$ works for all $\alpha > 1/2$ in Theorem~\ref{thm:T2};
this led us to add Remark~\ref{rem:ten} to the paper.
Also, we found that this run did not use the von Staudt Clausen theorem;
rather, it independently derived  \eqref{eqn:vsc_consequence} which it used to prove Theorem~\ref{thm:T2}.

We thus provide the corresponding Lean files as \texttt{extension/*.lean} in the above repository
as a computer-generated proof that Remark~\ref{rem:ten} is indeed true.

\end{document}